\documentclass[10pt]{amsart}
\usepackage{amsmath ,amsfonts, amssymb, graphics,amscd,eucal}
\usepackage[pdftex]{graphicx}
\usepackage[all]{xy}

\newcommand{\lp}{\left(}
\newcommand{\rp}{\right)}
\newcommand{\ZZ}{\mathbb{Z}}
\newcommand{\QQ}{\mathbb{Q}}

\newcommand{\CC}{\mathbb{C}}

\newcommand{\Der}{\textrm{Der}}

\newcommand{\pardiff}[2]{\frac{\partial #1}{\partial #2}}
\newcommand{\Var}{\textrm{Var}}
\newcommand{\GL}{\textrm{GL}}
\newcommand{\Poin}{\textrm{Poin}}

\newcommand{\pdim}{\textrm{pdim}}

\theoremstyle{plain}
\newtheorem{theorem}{Theorem}[section]

\newtheorem{lemma}[theorem]{Lemma}

\newtheorem{question}[theorem]{Question}

\newtheorem{remark}[theorem]{Remark}
\newtheorem{example}[theorem]{Example}

\newtheorem{definition}[theorem]{Definition}

\begin{document}

\title[Homological Properties of Determinantal Arrangements]
{Homological Properties of Determinantal Arrangements}
\author[Arnold Yim]{Arnold Yim}
\address{Department of Mathematics,
Purdue University,
150 North University Street,
West Lafayette, IN  47907-2067}
\email{ayim@purdue.edu}


\begin{abstract}
We explore a natural extension of braid arrangements in the context of determinantal arrangements. We show that these determinantal arrangements are free divisors. Additionally, we prove that free determinantal arrangements defined by the minors of $2\times n$ matrices satisfy nice combinatorial properties.

We also study the topology of the complements of these determinantal arrangements, and prove that their higher homotopy groups are isomorphic to those of $S^3$. Furthermore, we find that the complements of arrangements satisfying those same combinatorial properties above have Poincar\'e polynomials that factor nicely.
\end{abstract}



\thanks{I would like to express my gratitude to my advisor Uli Walther for his guidance throughout this whole project. Partial support by the NSF under grant DMS-1401392 is gratefully acknowledged.}

\maketitle
\setcounter{tocdepth}{1}
\tableofcontents
\numberwithin{equation}{section}

\section{Introduction}
Let $D$ be a divisor on an an $n$-dimensional complex analytic manifold $X$. The \emph{module of logarithmic derivations} $\Der_X(-\log D):=\{\theta\in \Der_X|\theta(\mathcal{O}_X(-Y))\subseteq \mathcal{O}_X(-Y)\}$ are the vector fields on $X$ that are tangent along $D$. If $\Der_X(-\log D)$ is locally free, then $D$ is called a \emph{free divisor}. The simplest example of free divisors are normal crossing divisors.

Free divisors were first introduced by Saito \cite{MR586450}, motivated by his study of the discriminants of versal deformations of
isolated hypersurface singularities. The study of free divisors coming from discriminants of versal deformations has since been a driving force in the theory of singularities (see \cite{MR747303, MR810963, MR735440, MR1334939, MR722502}).

Aside from versal deformations, free divisors show up naturally in many different settings. For example, many of the classically arising hyperplane arrangements are free (see \cite{MR1217488}). This includes braid arrangements and all Coxeter arrangements.

Surprisingly, freeness can also give us topological information. Specifically, Terao proves in \cite{MR608532} that for a free hyperplane arrangement, the Poincar\'e polynomial for the complement is determined by the degrees of the vector fields in the basis of the module of logarithmic derivations:

\begin{theorem}[Terao]
\label{TeraosTheorem}
Let $\mathcal{A}\subset \CC^n$ be a free central hyperplane arrangement and suppose that \emph{$\displaystyle\Der_{\CC^n}(-\log \mathcal{A})\cong\bigoplus_{i=1}^n \CC[x_1,\ldots,x_n](-b_i)$}, then \emph{$$\Poin(\CC^n\setminus\mathcal{A},t) = \prod_{i=1}^n (1+b_it).$$}
\end{theorem}

Observe that Poincar\'e polynomials are topological invariants that are not specific to hyperplane arrangements, and neither are the degrees of logarithmic vector fields for graded free divisors. Naturally, one might be interested in freeness for arrangements of more general hypersurfaces and how freeness might be connected to topology. For example, Schenck and Toh{\v{a}}neanu \cite{MR2495794} give conditions for when an arrangement of lines and conics on $\mathbb{P}^2$ is free.

We are particularly interested in \emph{determinantal arrangements}, which are configurations of determinantal varieties. Buchweitz and Mond \cite{MR2228227} showed that the arrangement defined by the product of the maximal minors of a $n\times(n+1)$ matrix of indeterminates is free. Recently, Damon and Pike \cite{MR2868903} show that certain determinantal arrangements coming from symmetric, skew-symmetric and square general matrices are free and have complements that are $K(\pi,1)$. In both of these cases, the arrangements turn out to be linear free divisors (i.e. the basis for $\Der_X(-\log D)$ is generated by linear vector fields). The vector fields arising in these situations correspond to matrix group actions on the generic matrix which stabilize the divisor $D$. Many interesting determinantal arrangements, however, are not linear free divisors as our next example shows.

\begin{example}
\label{2by4example}
\emph{Let $M$ be the $2\times 4$ matrix of indeterminates
$$M= \left(
       \begin{array}{cccc}
         x_1 & x_2 & x_3 & x_4 \\
         y_1 & y_2 & y_3 & y_4 \\
       \end{array}
     \right),$$
and for $i<j$, let $\Delta_{ij}$ be the $2$-minor of $M$ using the $i$-th and $j$-th columns, $\Delta_{ij} = x_iy_j-x_jy_i$. Let $f$ be the product $f= \displaystyle\prod_{i<j} \Delta_{ij}$. Then $\Der_{X}(-\log f)$ is free with basis consisting of $7$ linear derivations (coming from $\mathrm{SL}(2,\CC)$-action, column-scaling, and row-scaling on $M$), and one derivation of degree $5$: $\theta = \Delta_{24}\Delta_{34}\left(x_1\pardiff{}{x_4}+y_1\pardiff{}{y_4}\right)$.}
\end{example}

In this paper, we study the determinantal arrangement analog of the braid arrangement. Given a $2\times n$ matrix of indeterminates, we define determinantal arrangements by taking products of its maximal minors. In Theorem \ref{CompleteDetArr}, we show that the arrangement defined by taking the product of all maximal minors is free. Furthermore, we prove in Theorem \ref{ChordalDetArr} that free determinantal arrangements satisfy certain combinatorial properties. In Theorems \ref{CompletePoincare} and \ref{ChordalPoincare}, we show that the Poincar\'e polynomial of the complement of a free determinantal arrangement factors nicely.

\section{Setup}

We look at divisors on $X=\CC^{2n}$ with coordinate ring $R= \CC[x_1,\ldots,x_n, y_1, \ldots, y_n]$. Let $\Der_X$ be the free $R$-module of vector fields on $X$ generated by $\left\{\pardiff{}{x_i},\pardiff{}{y_i}\right\}_{i=1..n}$. For any divisor $f$ on $X$, we are interested in the following object:

\begin{definition} \emph{The} module of logarithmic derivations \emph{along $f$ is the $R$-module $$\Der_X(-\log f) = \{\theta \in \Der_X | \theta(f) \in (f)\}.$$}
\end{definition}

We want to know when $f$ has a well-behaved singular locus, thus we are interested in when the module of logarithmic derivations along $f$ is free. We say:

\begin{definition} \emph{A divisor $f$ on $X$ is} free \emph{if $\Der_X(-\log f)$ is a free $R$-module.}
\end{definition}

To determine whether a divisor is free, we use \emph{Saito's criterion} \cite{MR586450}:

\begin{theorem}[Saito]
\label{SaitosCriterion}
A divisor $f\in \CC[x_1,\ldots,x_n]$ is a free divisor if and only if there exists $n$ elements \emph{$$\theta_j = \displaystyle\sum_{i=1}^n g_{ij}\pardiff{}{x_i} \in \Der_X(-\log f)$$} such that $\det((g_{ij})) = c\cdot f$ for some non-zero $c\in \CC$.
\end{theorem}

We focus on logarithmic derivations for hypersurface arrangements defined by graphs. In the context of of hyperplane arrangements, these are called graphic arrangements. Given a graph $G$ with $n$ vertices, we associate a hyperplane arrangement defined by a polynomial $f\in \CC[x_1,\ldots x_n]$. For each edge of $G$ between vertices $v_i$ and $v_j$, we include the hyperplane defined by $x_i-x_j=0$ in the arrangement. For example, the graphic arrangement associated to a complete graph on $n$ vertices is the braid arrangement on $n$ variables defined by $f=\displaystyle\prod_{1\leq i<j\leq n} (x_i-x_j)$.

Due to a result by Stanley \cite{MR0309815}, one knows that a graphic arrangement is free if and only if its corresponding graph is chordal (i.e. a graph for which every cycle of length greater than $3$ has a chord). More recently, Kung and Schenck \cite{MR2260017} improved this result and found that $\pdim(\Der(-\log f))\geq k-3$ where $f$ defines a graphic arrangement with longest chord-free induced cycle of length $k$.  We will be using a characterization of chordality given by Fulkerson and Gross \cite{MR0186421}:
\begin{definition}
\emph{A graph $G$ is} chordal \emph{if and only if there exists an ordering of vertices, such that for each vertex $v$, the induced subgraph on $v$ and its neighbors that occur before it in the sequence is a complete graph.}
\end{definition}

While freeness is well understood for graphic arrangements, it is still unclear when we consider arrangements of more general hypersurfaces. We investigate certain determinantal arrangements associated to graphs. Specifically, let $M$ be the $2\times n$ matrix of indeterminates
$$M= \left(
       \begin{array}{cccc}
         x_1 & x_2 & \cdots & x_n \\
         y_1 & y_2 & \cdots & y_n \\
       \end{array}
     \right).$$ For $i<j$, let $\Delta_{ij}$ denote the $2$-minor of $M$ using the $i$-th and $j$-th columns, $\Delta_{ij}=x_iy_j-x_jy_i$.

\begin{definition}
\emph{For each graph $G$ with $n$ vertices, we can associate a} determinantal arrangement, \emph{$\mathcal{A}_G$, consisting of the determinantal varieties $\Var(\Delta_{ij})$ for each edge between vertices $v_i$ and $v_j$ of $G$.}
\end{definition}

\section{Freeness of Determinantal Arrangements}

For hyperplane arrangements, the braid arrangement is a well-known example of an arrangement that is free. The braid arrangement is made up of hyperplanes that are defined by any two coordinates being equal. We define something similar in our setting. Consider $\CC^{2n}$ as a collection of $n$ two-dimensional vectors, and thus the hypersurface defined by the vanishing of a minor of $M$ is the hypersurface defined by two vectors being linearly dependent. Our analog of the braid arrrangement is the determinantal arrangement defined by any two columns being linearly dependent. If we think about these arrangements as coming from graphs, both the braid arrangement and our analog come from the complete graph on $n$ vertices.

In Theorem \ref{CompleteDetArr}, we prove that our analog of the braid arrangement is free: we construct a generating set for the module of logarithmic derivations and show that this set satisfies Saito's criterion. In Theorem \ref{ChordalDetArr}, we prove that if a graph is not chordal, then the corresponding determinantal arrangement is not free. We show that near a particular point, our arrangement looks like the cyclic graphic arrangement which has projective dimension related to the length of the cycle.


Before proving Theorem \ref{CompleteDetArr}, we will need the two following lemmas:

\begin{lemma}
\label{SymPolyDet}
For $n\in \ZZ_{>0}$, let $s_{i,j,k}$ denote the degree $k$ symmetric polynomial on the variables $z_i,\ldots, z_n$ that is linear in each variable omitting the variable $z_j$, given by $$s_{i,j,k} = \displaystyle \sum_{\begin{array}{c}
                                                                      \alpha_m\neq j \\
                                                                      i\leq \alpha_1<\cdots < \alpha_k \leq n
                                                                    \end{array}} z_{\alpha_1}z_{\alpha_2} \cdots z_{\alpha_k},$$ and let $s_{i,j,0}=1$.

Let $A_i$ denote the $(n+1-i)\times(n+1-i)$ matrix $(s_{i,j,k})$, where the row index $j$ ranges from $i$ to $n$, and the column index $k$ ranges from $0$ to $n-i$. Then $$\det(A_i) = \lp\displaystyle\prod_{i<s\leq n } (z_i-z_s)\rp \det(A_{i+1}).$$\end{lemma}

\begin{proof}
Writing out $A_i$, we have
$$A_i=\left(
      \begin{array}{cccc}
        1 & (z_{i+1}+z_{i+2}+\cdots+ z_n) & \cdots & (z_{i+1}z_{i+2}\cdots z_n) \\
        1 & (z_i+z_{i+2}+\cdots+z_n) & \cdots & (z_iz_{i+2}\cdots z_n) \\
        \vdots & \vdots & \ddots & \vdots \\
        1 & (z_{i}+z_{i+1}+\cdots + z_{n-1}) & \cdots & (z_iz_{i+1}\cdots z_{n-1}) \\
      \end{array}
    \right).$$
    Subtracting the first row from every other row, we have
    $$\left(
        \begin{array}{ccccc}
          1 & (z_{i+1}+z_{i+2}+\ldots+ z_n) & (z_{i+1}z_{i+2} + z_{i+1}z_{i+3}+\cdots + z_{n-1}z_n) & \cdots & (z_{i+1}z_{i+2}\cdots z_n) \\
          0 & (z_i-z_{i+1}) & (z_i-z_{i+1})(z_{i+2}+z_{i+3}+\cdots+ z_n) & \cdots & (z_{i}-z_{i+1})(z_{i+2}z_{i+3}\cdots z_n) \\
          0 & (z_i-z_{i+2}) & (z_i-z_{i+2})(z_{i+1}+z_{i+3}+\cdots + z_n) & \cdots & (z_i-z_{i+2})(z_{i+1}z_{i+3}\cdots z_n) \\
          \vdots & \vdots & \vdots & \ddots & \vdots \\
          0 & (z_i-z_n) & (z_i-z_n)(z_{i+1}+z_{i+2}+\cdots+z_{n-1}) & \cdots & (z_i-z_n)(z_{i+1}z_{i+2}\cdots z_{n-1}) \\
        \end{array}
      \right).$$
      We can factor the lower right $(n-i)\times (n-i)$ submatrix as
      $$\left(
        \begin{array}{cccc}
          (z_i-z_{i+1}) &  &  &  \\
           & (z_{i}-z_{i+2}) &  & \\
           &  & \ddots &  \\
           &  &  & (z_i-z_n) \\
        \end{array}
      \right)\left(
        \begin{array}{cccc}
           1 & (z_{i+2}+z_{i+3}+\cdots+ z_n) & \cdots & (z_{i+2}z_{i+3}\cdots z_n) \\
           1 & (z_{i+1}+z_{i+3}+\cdots + z_n) & \cdots & (z_{i+1}z_{i+3}\cdots z_n) \\
           \vdots & \vdots & \ddots & \vdots \\
          1 & (z_{i+1}+z_{i+2}+\cdots+z_{n-1}) & \cdots & (z_{i+1}z_{i+2}\cdots z_{n-1}) \\
        \end{array}
      \right)
      $$
      $$=\left(
        \begin{array}{cccc}
          (z_i-z_{i+1}) &  &  &  \\
           & (z_{i}-z_{i+2}) &  & \\
           &  & \ddots &  \\
           &  &  & (z_i-z_n) \\
        \end{array}
      \right) A_{i+1},$$
      thus $\det(A_i)=\lp\displaystyle\prod_{i<s\leq n } (z_i-z_s)\rp \det(A_{i+1})$.
\end{proof}

\begin{lemma}
\label{BlockDet}
Let $A$ be a block matrix $A=\left(
                               \begin{array}{cc}
                                 A_1 & A_2 \\
                                 A_3 & A_4 \\
                               \end{array}
                             \right)$
 with blocks of size $n\times n$ with entries in $\CC(z_1,\ldots,z_n)$. If $A_1$ and $A_3$ are diagonal matrices with nonzero entries, then $\det(A) = \det(A_1A_4-A_3A_2)$.
\end{lemma}
\begin{proof}
Let $B$ be the block matrix $B=\left(
                               \begin{array}{cc}
                                 A_1^{-1} & 0 \\
                                 0 & A_3^{-1} \\
                               \end{array}
                             \right)$, then $BA = \left(
                               \begin{array}{cc}
                                 I_n & A_1^{-1}A_2 \\
                                 I_n & A_3^{-1}A_4 \\
                               \end{array}
                             \right)$. Using row reduction, we find
$$\det(BA) = \det\left(
                               \begin{array}{cc}
                                 I_n & A_1^{-1}A_2 \\
                                 0 & A_3^{-1}A_4-A_1^{-1}A_2 \\
                               \end{array}
                             \right).$$
Now, let $C$ be the block matrix $C=\left(
                               \begin{array}{cc}
                                 I_n & 0 \\
                                 0 & A_1A_3 \\
                               \end{array}
                             \right)$, then
$$\det(CBA) = \det\left(
                               \begin{array}{cc}
                                 I_n & A_1^{-1}A_2 \\
                                 0 & A_1A_4-A_3A_2 \\
                               \end{array}
                             \right) = \det(A_1A_4-A_3A_2).$$
Since $\det(CBA)=\det(A)$, we have $\det(A)=  \det(A_1A_4-A_3A_2)$.
\end{proof}

Now, we have our main result of this section:

\begin{theorem}
\label{CompleteDetArr}
Let $G$ be the complete graph on $n$ vertices for $n\geq 3$. The determinantal arrangement $\mathcal{A}_G$ is free.
  \end{theorem}

\begin{proof}
If $G$ is the complete graph on $n$ vertices, then the corresponding determinantal arrangement $\mathcal{A}_G$ is defined by $$f= \prod_{1\leq i<j \leq n} \Delta_{ij}=\prod_{1\leq i<j \leq n} x_iy_j-x_jy_i.$$
We provide a set of elements in $\Der_X(-\log f)$, and show that the this set actually forms a basis for $\Der_X(-\log f)$ according to Saito's criterion.

We first consider several linear derivations: $$ \alpha = \displaystyle\sum_{k=1}^n x_k \pardiff{}{y_k}$$
$$\beta = \displaystyle \sum_{k=1}^n y_k \pardiff{}{x_k}$$
$$\gamma = \displaystyle \sum_{k=1}^n y_k \pardiff{}{y_k}.$$
To show that these derivations belong to $\Der_X(-\log f)$, we show that they stabilize the ideal of each minor, and thus they stabilize the ideal of the product of the minors:
$$\begin{array}{ccl}
                            \alpha(\Delta_{ij}) & = & \lp\displaystyle\sum_{k=1}^n x_k\pardiff{}{y_k}\rp\lp x_iy_j-x_jy_i\rp \\
                             & = & \lp x_i\pardiff{}{y_i} + x_j\pardiff{}{y_j}\rp \lp x_iy_j-x_jy_i\rp  \\
                             & = & -x_ix_j+x_jx_i \\
                             & = & 0.
                          \end{array}$$ Since $\alpha$ stabilizes each $(\Delta_{ij})$, $\alpha\in \Der_X(-\log f)$.

Similarly, $$\begin{array}{ccl}
                            \beta(\Delta_{ij}) & = & \lp\displaystyle\sum_{k=1}^n y_k\pardiff{}{x_k}\rp\lp x_iy_j-x_jy_i\rp \\
                             & = & \lp y_i\pardiff{}{x_i} + y_j\pardiff{}{x_j}\rp \lp x_iy_j-x_jy_i\rp  \\
                             & = & y_iy_j-y_jy_i \\
                             & = & 0,
                          \end{array}$$
and $$\begin{array}{ccl}
                            \gamma(\Delta_{ij}) & = & \lp\displaystyle\sum_{k=1}^n y_k\pardiff{}{y_k}\rp\lp x_iy_j-x_jy_i\rp \\
                             & = & \lp y_i\pardiff{}{y_i} + y_j\pardiff{}{y_j}\rp \lp x_iy_j-x_jy_i\rp  \\
                             & = & -y_ix_j+y_jx_i \\
                             & = & \Delta_{ij},
                          \end{array}$$
                          thus $\beta, \gamma \in \Der_X(-\log f)$.

We also have $n$ linear derivations $$\theta_k = x_k\pardiff{}{x_k} + y_k\pardiff{}{y_k}$$ for $k=1, 2\ldots, n$. We have $$\begin{array}{ccl}
                            \theta_k(\Delta_{kj}) & = & \lp x_k\pardiff{}{x_k} + y_k\pardiff{}{y_k}\rp\lp x_ky_j-x_jy_k\rp \\
                             & = &x_ky_j-y_kx_j  \\
                             & = & \Delta_{kj},
                          \end{array}$$ and similarly, $\theta_k(\Delta_{ik}) = \Delta_{ik}$. When $i,j\neq k$, $\theta_k(\Delta_{ij}) = 0$, thus $\theta_k$ stabilizes each ($\Delta_{ij})$. This shows that $\theta_k \in \Der_X(-\log f)$.

Finally, we have $n-3$ elements of degree $n+1$.
For $k=4,5,..,n$, let $\tau_k$ be a bijection of sets from $\{1,\ldots, n-4\}$ to $\{4,\ldots, k-1, k+1,\ldots n\}$, and let $S_{n-4}$ be the symmetric group on the numbers $\{1,\ldots, n-4\}$. For $m=0,1,\ldots, n-4$, define $$ a_{m,k}=\frac{1}{m!(n-4-m)!}\sum_{\sigma\in S_{n-4}} x_{(\tau_k\circ \sigma)(1)}\cdots x_{(\tau_k\circ \sigma)(m)}y_{(\tau_k\circ \sigma)(m+1)}\cdots y_{(\tau_k\circ \sigma)(n-4)}.$$ Now, consider the derivations $$\varphi_m =\displaystyle \sum_{k=4}^n a_{m,k} \Delta_{2k}\Delta_{3k} \lp x_1\pardiff{}{x_k} + y_1\pardiff{}{y_k}\rp.$$

If $i,j<4$, then $\varphi_m(\Delta_{ij}) = 0$. Now, suppose that $i<4$ and $j\geq 4$, then $$\begin{array}{ccl}
\varphi_m(\Delta_{ij}) & = & \lp\displaystyle\sum_{k=4}^n a_{m,k} \Delta_{2k}\Delta_{3k} \lp x_1\pardiff{}{x_k} + y_1\pardiff{}{y_k}\rp\rp\lp x_iy_j-x_jy_i\rp \\
                             & = &a_{m,j}\Delta_{2j}\Delta_{3j}\lp x_1\pardiff{}{x_j} + y_1\pardiff{}{y_j}\rp \lp x_iy_j-x_jy_i\rp \\
                             & = & a_{m,j}\Delta_{2j}\Delta_{3j}\lp -x_1y_i+y_1x_i\rp.
                          \end{array}$$ When $i=2,3$, $\varphi_m(\Delta_{ij})\in (\Delta_{ij})\subseteq R$, and when $i=1$, $\varphi_m(\Delta_{ij})=0$.

If $i,j\geq 4$, $$\begin{array}{ccl}
\varphi_m(\Delta_{ij}) & = & \lp\displaystyle\sum_{k=4}^n a_{m,k} \Delta_{2k}\Delta_{3k} \lp x_1\pardiff{}{x_k} + y_1\pardiff{}{y_k}\rp\rp\lp x_iy_j-x_jy_i\rp \\
                             & = &\lp a_{m,i}\Delta_{2i}\Delta_{3i}\lp x_1\pardiff{}{x_i} + y_1\pardiff{}{y_i}\rp + a_{m,j}\Delta_{2j}\Delta_{3j}\lp x_1\pardiff{}{x_j} + y_1\pardiff{}{y_j}\rp\rp \lp x_iy_j-x_jy_i\rp \\
                             & = & a_{m,i}\Delta_{2i}\Delta_{3i}\lp x_1y_j - y_1x_j\rp + a_{m,j}\Delta_{2j}\Delta_{3j}\lp -x_1y_i+y_1x_i\rp\\
                             & = & a_{m,i}\Delta_{2i}\Delta_{3i}\Delta_{1j} - a_{m,j}\Delta_{2j}\Delta_{3j}\Delta_{1i}.
                          \end{array}$$
Note that each term in $a_{m,i}$ either has a factor of $x_j$ or $y_j$, and also note that the terms in $a_{m,j}$ are exactly the terms in $a_{m,i}$, with $x_i$ and $y_i$ instead of $x_j$ and $y_j$ respectively, thus it is enough to show that $x_j\Delta_{2i}\Delta_{3i}\Delta_{1j} - x_i\Delta_{2j}\Delta_{3j}\Delta_{1i}$ and $y_j\Delta_{2i}\Delta_{3i}\Delta_{1j} - y_i\Delta_{2j}\Delta_{3j}\Delta_{1i}$ are divisible by $\Delta_{ij}$. Using Pl\"ucker relations, we can write:
$$\begin{array}{ccl}
    x_j\Delta_{2i}\Delta_{3i}\Delta_{1j} - x_i\Delta_{2j}\Delta_{3j}\Delta_{1i} & = & x_j\Delta_{3i}(\Delta_{1j}\Delta_{2i}) -x_i \Delta_{2j}\Delta_{3j}\Delta_{1i} \\
    & = & x_j\Delta_{3i}(\Delta_{1i}\Delta_{2j}-\Delta_{12}\Delta_{ij}) -x_i \Delta_{2j}\Delta_{3j}\Delta_{1i} \\
     & = & \Delta_{1i}\Delta_{2j} (x_j\Delta_{3i}-x_i\Delta_{3j}) - x_j \Delta_{3i}\Delta_{12}\Delta_{ij} \\
     & = &  \Delta_{1i}\Delta_{2j}(x_jx_3y_i-x_jx_iy_3 - x_ix_3y_j + x_ix_jy_3)- x_j \Delta_{3i}\Delta_{12}\Delta_{ij}\\
     & = & \Delta_{1i}\Delta_{2j}(x_jx_3y_i- x_ix_3y_j)- x_j \Delta_{3i}\Delta_{12}\Delta_{ij}\\
     & = & \Delta_{1i}\Delta_{2j}(-x_3\Delta_{ij}) - x_j\Delta_{3i}\Delta_{12}\Delta_{ij} \in (\Delta_{ij}),
  \end{array}$$
and similarly,
$$\begin{array}{ccl}
    y_j\Delta_{2i}\Delta_{3i}\Delta_{1j} - y_i\Delta_{2j}\Delta_{3j}\Delta_{1i} & = & y_j\Delta_{3i}(\Delta_{1j}\Delta_{2i}) -y_i \Delta_{2j}\Delta_{3j}\Delta_{1i} \\
    & = & y_j\Delta_{3i}(\Delta_{1i}\Delta_{2j}-\Delta_{12}\Delta_{ij}) -y_i \Delta_{2j}\Delta_{3j}\Delta_{1i} \\
     & = & \Delta_{1i}\Delta_{2j} (y_j\Delta_{3i}-y_i\Delta_{3j}) - y_j \Delta_{3i}\Delta_{12}\Delta_{ij} \\
     & = &  \Delta_{1i}\Delta_{2j}(y_jx_3y_i-y_jx_iy_3 - y_ix_3y_j + y_ix_jy_3)- y_j \Delta_{3i}\Delta_{12}\Delta_{ij}\\
     & = & \Delta_{1i}\Delta_{2j}(-y_jx_iy_3 + y_ix_jy_3)- x_j \Delta_{3i}\Delta_{12}\Delta_{ij}\\
     & = & \Delta_{1i}\Delta_{2j}(-y_3\Delta_{ij}) - x_j\Delta_{3i}\Delta_{12}\Delta_{ij} \in (\Delta_{ij}).
  \end{array}$$ Since $\varphi_m$ stabilizes each $(\Delta_{ij})$, $\varphi_m\in \Der_X(-\log f)$.

It remains to show that this set of elements in $\Der_X(-\log f)$ form a basis. According to Saito's criterion, these derivations form a basis if and only if the determinant of the coefficient matrix is a nonzero constant multiple of $f$. With our elements, we have the coefficient matrix:
$$\left(
  \begin{array}{cccccccccccc}
    y_1 &  &  & x_1 &  &  &  &  &  & & & \\
    y_2 &  &  &  & x_2 &  &  &  &  & & & \\
    y_3 &  &  &  &  & x_3 &  &  &  & & & \\
    y_4 &  &  &  &  &  & x_4 &  &  & a_{0,4}\Delta_{24}\Delta_{34}x_1 & \cdots & a_{n-4,4}\Delta_{24}\Delta_{34}x_1 \\
    \vdots & & & & & & & \ddots & & \vdots& \ddots &\vdots\\
    y_n& & & & & & & & x_n& a_{0,n}\Delta_{2n}\Delta_{3n}x_1&\cdots & a_{n-4,n}\Delta_{2n}\Delta_{3n}x_1\\
     & x_1 & y_1 & y_1 &  &  &  &  &  & & &  \\
     & x_2 & y_2 &  & y_2 & &  &  &  & & & \\
     &  x_3  & y_3 &  &  &y_3  &  &  & & &  \\
     & x_4 & y_4 &  &  &  & y_4 &  &  & a_{0,4}\Delta_{24}\Delta_{34}y_1 &\cdots &a_{n-4,4}\Delta_{24}\Delta_{34}y_1 \\
         &\vdots & \vdots & & & & & \ddots & & \vdots &\ddots &\vdots\\
    &x_n &y_n & & & & & & y_n& a_{0,n}\Delta_{2n}\Delta_{3n}y_1&\cdots &a_{n-4,n}\Delta_{2n}\Delta_{3n}y_1\\
  \end{array}
\right).$$

We swap some rows to organize our matrix into blocks (this could potentially change the determinant by a sign, but that is unimportant in checking Saito's criterion):
$$\left(
  \begin{array}{cccccc|cccccc}
    y_1 &  &  & x_1 &  &  &  &  &  & & & \\
    y_2 &  &  &  & x_2 &  &  &  &  & & & \\
    y_3 &  &  &  &  & x_3 &  &  &  & & & \\
     & x_1 & y_1 & y_1 &  &  &  &  &  & & &  \\
     & x_2 & y_2 &  & y_2 & &  &  &  & & & \\
     &  x_3  & y_3 &  &  &y_3  &  &  & & &  \\\hline
        y_4 &  &  &  &  &  & x_4 &  &  & a_{0,4}\Delta_{24}\Delta_{34}x_1 & \cdots & a_{n-4,4}\Delta_{24}\Delta_{34}x_1 \\
    \vdots & & & & & & & \ddots & & \vdots& \ddots &\vdots\\
    y_n& & & & & & & & x_n& a_{0,n}\Delta_{2n}\Delta_{3n}x_1&\cdots & a_{n-4,n}\Delta_{2n}\Delta_{3n}x_1\\
     & x_4 & y_4 &  &  &  & y_4 &  &  & a_{0,4}\Delta_{24}\Delta_{34}y_1 &\cdots &a_{n-4,4}\Delta_{24}\Delta_{34}y_1 \\
         &\vdots & \vdots & & & & & \ddots & & \vdots &\ddots &\vdots\\
    &x_n &y_n & & & & & & y_n& a_{0,n}\Delta_{2n}\Delta_{3n}y_1&\cdots &a_{n-4,n}\Delta_{2n}\Delta_{3n}y_1\\
  \end{array}
\right).$$
Denote the matrix above by $N$, with blocks $N=\left(
    \begin{array}{c|c}
      A & 0 \\\hline
      C & D \\
    \end{array}
  \right)$. Since $N$ is a triangular block matrix, $\det(N) = \det(A)\det(D)$. By explicit computation, we find that \begin{equation}\det(A) = \Delta_{12}\Delta_{13}\Delta_{23} \label{eq:detA}.\end{equation} To calculate the determinant of $D$, we split the matrix into more blocks: $$ D = \left(
    \begin{array}{c|c}
      D_1 & D_2 \\\hline
      D_3 & D_4 \\
    \end{array}
  \right) = \left(
              \begin{array}{ccc|ccc}
                 x_4 &  &  & a_{0,4}\Delta_{24}\Delta_{34}x_1 & \cdots & a_{n-4,4}\Delta_{24}\Delta_{34}x_1 \\
               & \ddots & & \vdots& \ddots &\vdots\\
               & & x_n& a_{0,n}\Delta_{2n}\Delta_{3n}x_1&\cdots & a_{n-4,n}\Delta_{2n}\Delta_{3n}x_1 \\ \hline
                y_4 &  &  & a_{0,4}\Delta_{24}\Delta_{34}y_1 &\cdots &a_{n-4,4}\Delta_{24}\Delta_{34}y_1\\
                &\ddots & & \vdots& \ddots &\vdots \\
                & & y_n& a_{0,n}\Delta_{2n}\Delta_{3n}y_1&\cdots &a_{n-4,n}\Delta_{2n}\Delta_{3n}y_1\\
              \end{array}
            \right).$$


  By Lemma \ref{BlockDet}, we have $\det(D)=\det(D_1D_4-D_3D_2)$. Now,
$$\begin{array}{ccl}
    D_1D_4 - D_3D_2 & = & \left(
                      \begin{array}{ccc}
                        a_{0,4}\Delta_{24}\Delta_{34}(y_1x_4-x_1y_4) & \cdots & a_{n-4,4}\Delta_{24}\Delta_{34}(y_1x_4-x_1y_4) \\
                        \vdots &
                        \ddots & \vdots \\
                       a_{0,n}\Delta_{2n}\Delta_{3n}(y_1x_n-x_1y_n)&\cdots & a_{n-4,n}\Delta_{2n}\Delta_{3n}(y_1x_n-x_1y_n) \\
                      \end{array}
                    \right) \\
     & = & -\left(
                      \begin{array}{ccc}
                        a_{0,4}\Delta_{24}\Delta_{34}\Delta_{14} & \cdots & a_{n-4,4}\Delta_{24}\Delta_{34}\Delta_{14} \\
                        \vdots &
                        \ddots & \vdots \\
                       a_{0,n}\Delta_{2n}\Delta_{3n}\Delta_{1n}&\cdots & a_{n-4,n}\Delta_{2n}\Delta_{3n}\Delta_{1n} \\
                      \end{array}
                    \right) \\
     & = & -\left(
              \begin{array}{ccc}
                \Delta_{24}\Delta_{34}\Delta_{14} &  &  \\
                 & \ddots &  \\
                 &  & \Delta_{2n}\Delta_{3n}\Delta_{1n} \\
              \end{array}
            \right)
     \left(
                      \begin{array}{ccc}
                        a_{0,4} & \cdots & a_{n-4,4} \\
                        \vdots &
                        \ddots & \vdots \\
                       a_{0,n}&\cdots & a_{n-4,n} \\
                      \end{array}
                    \right)\\
    & =: & -D_5 D_6.
  \end{array}$$
  Observe that \begin{equation}\label{eq:detD5} \det(D_5)= \displaystyle\prod_{i=1}^3\prod_{j=4}^n \Delta_{ij},\end{equation} therefore it remains to show that $\det(D_6)$ is a nonzero constant multiple of the product of all minors using the last $n-3$ columns of $M$.

  We show that each $\Delta_{ij}$ for $i,j\geq 4$ divides $\det(D_6)$ by showing that $\det(D_6)$ vanishes on $\Var(\Delta_{ij})$. Indeed, $\Delta_{ij}$ vanishes when columns $i$ and $j$ of $M$ are scalar multiples of each other. Write $x_j = cx_i$ and $y_j = cy_i$. Looking to rows $i$ and $j$ of $D_6$, we have $a_{m,j}=ca_{m,i}$, and since these rows are scalar multiples of each other, $\det(D_6)$ vanishes here which implies that each $\Delta_{ij}$ divides $\det(D_6)$. The degree of the product of the minors, $2\left(
                                                                                                                 \begin{array}{c}
                                                                                                                   n-3 \\
                                                                                                                   2 \\
                                                                                                                 \end{array}
                                                                                                               \right) = (n-3)(n-4),$
  is the same as the degree of $\det(D_6)$, hence $\det(D_6)$ is a constant multiple of the product of the minors. To check that $\det(D_6)$ is not identically zero, we substitute $y_k=1$ into $D_6$ to get the matrix in Lemma \ref{SymPolyDet} on the variables $x_4,\ldots, x_n$, thus if $x_4\neq x_5\neq \cdots \neq x_n$, then $\det(D_6)\neq 0$.

 With equations (\ref{eq:detA}) and (\ref{eq:detD5}), we find $\det(N)= (-1)^{n-3}\det(A)\det(D_5)\det(D_6)$ is a constant multiple of the product of all of the minors of $M$. By Saito's criterion, $\{\alpha, \beta, \gamma, \theta_1, \ldots, \theta_n, \varphi_0,\ldots, \varphi_{n-4}\}$ form a basis for $\Der_X(-\log f)$, hence our determinantal arrangement is free.
\end{proof}

We believe that our work with determinantal arrangements on $2\times n$ generic matrices only scratches the surface of a broader class of free divisors. For example, we can change the size of our generic matrix. In the case where $m=3$ and $n=4$, one knows that the arrangement is a linear free divisor (see \cite{MR2228227}, \cite{MR2521436}). However, in the next case, $m=3$ and $n=5$, we already don't know whether or not the arrangement is free. More generally, one can ask:

\begin{question}
\emph{Let $M$ be the $m\times n$ matrix of indeterminates with $n>m>2$, and let $f$ be the product of all maximal minors of $M$. Is the arrangement defined by $f$ free?}
\end{question}

One can also consider determinantal arrangements defined by subgraphs of the complete graph. Much like hyperplane arrangements, we find that the freeness of the determinantal arrangement is related to whether or not the graph is chordal.

\begin{theorem}
\label{ChordalDetArr}
If a determinantal arrangement $\mathcal{A}_G$ is free, then $G$ is chordal. Moreover, if $G$ has a chord-free induced cycle of length $k$, then \emph{$$\pdim(\Der_X(-\log \mathcal{A}_G))\geq k-3.$$}
\end{theorem}

\begin{proof}
Suppose that $G$ is not chordal, then $G$ has an chord-free induced cycle of length $k$ where $4\leq k \leq n$. We can reorganize the columns of $M$ so that this chord-free induced cycle occurs on the first $k$ vertices of $\mathcal{A}_G$. To show that $\mathcal{A}$ is not free, we will localize to a neighborhood of the point $p=\left(
      \begin{array}{ccccccc}
        1 & \cdots & 1 & 1 & 1 &  \cdots & 1  \\
        0 & \cdots & 0 & 1 & 2 & \cdots & n-k\\
      \end{array}
    \right)$. We will consider our divisor in the local ring $\CC[x_1,\ldots,x_n,y_1,\ldots,y_n]_{\mathfrak{m}_p}$ where $\mathfrak{m}_p$ is the maximal ideal associated to the point $p$. In this local ring, $\Delta_{ij}$ is a unit if $i$ or $j$ is greater than $k$. Thus, around $p$, $\mathcal{A}_G$ looks like $\Var(\Delta_{12}\Delta_{23}\cdots\Delta_{(k-1)k}\Delta_{1k})$ whose associated graph is the cyclic graph on $k$ vertices.

    We show that $p$ is in the non-free locus of $\Var(\Delta_{12}\Delta_{23}\cdots\Delta_{(k-1)k}\Delta_{1k})$. In our local ring, $x_i$ is a unit for all $i$, thus $\Var(\Delta_{12}\Delta_{23}\cdots\Delta_{(k-1)k}\Delta_{1k}) = \Var\left(\frac{x_1^{k-2}x_k^{k-2}}{x_2^2x_3^2\cdots x_{k-1}^2}\Delta_{12}\Delta_{23}\cdots\Delta_{(k-1)k}\Delta_{1k}\right)$. But, $$
    \begin{array}{l}
      \frac{x_1^{k-2}x_k^{k-2}}{x_2^2x_3^2\cdots x_{k-1}^2}\Delta_{12}\Delta_{23}\cdots\Delta_{(k-1)k}\Delta_{1k} \\
    \begin{array}{ccc}
       & = & \frac{x_1^{k-2}x_k^{k-2}}{x_2^2x_3^2\cdots x_{k-1}^2}(x_1y_2-x_2y_1)(x_2y_3-x_3y_2)\cdots (x_{k-1}y_k-x_ky_{k-1})(x_1y_k-x_ky_1) \\
       & = & \left(\frac{x_1x_k}{x_2}y_2-x_ky_1\right)\left(\frac{x_1x_k}{x_3}y_3-\frac{x_1x_k}{x_2}y_2\right)\cdots\left(x_1y_k-\frac{x_1x_k}{x_{k-1}}y_{k-1}\right)(x_1y_k-x_ky_1).
    \end{array}
    \end{array}$$

    Now, making a change of coordinates $$\begin{array}{ccc}
                                            z_1 & \leftrightarrow & x_ky_1 \\
                                            z_2 & \leftrightarrow & \frac{x_1x_k}{x_2}y_2 \\
                                            \vdots & \vdots & \vdots \\
                                            z_{k-1} & \leftrightarrow & \frac{x_1x_k}{x_{k-1}}y_{k-1} \\
                                            z_k & \leftrightarrow& x_1y_k
                                          \end{array},$$ we have that $\Var(\Delta_{12}\Delta_{23}\cdots\Delta_{(k-1)k}\Delta_{1k}) = \Var((z_2-z_1)(z_3-z_2)\cdots(z_k-z_{k-1})(z_k-z_1)).$ Since our point $p$, corresponds to $z_i=0$ for the cyclic graphic arrangement $\Var((z_2-z_1)(z_3-z_2)\cdots(z_k-z_{k-1})(z_k-z_1))$, we know that $p$ is in the non-free locus of $\Var(\Delta_{12}\Delta_{23}\cdots\Delta_{(k-1)k}\Delta_{1k})$, and thus $\mathcal{A}_G$ is not free. Moreover, this is a generic hyperplane arrangement so by Rose and Terao \cite{MR1089305}, $\pdim(\Der_X(-\log (\Delta_{12}\Delta_{23}\cdots\Delta_{(k-1)k}\Delta_{1k})))=k-3$. Since localization is an exact functor, $\pdim(\Der_X(-\log \mathcal(A)_G))\geq k-3$.
\end{proof}

\begin{remark}
\emph{The converse of Theorem \ref{ChordalDetArr} is not exactly true. For example, for any chordal graph with a vertex $v$ of degree 2, if the induced subgraph $v$ with its neighbors is not a cycle then the corresponding determinantal arrangement is not free. In this case, the arrangement locally behaves like $f=\Delta_{12}\Delta_{13}$, and one can check that this arrangement is not free. However, evidence suggests that many of the arrangements with chordal graphs are indeed free. For example, arrangements corresponding to doubly-connected (graphs that remain connected after removing any single vertex) chordal graphs seem to be free.}
\end{remark}

\section{Complements of Determinantal Arrangements}
Terao's theorem (\cite{MR608532}) relating the Poincar\'e polynomial for the complement of a free hyperplane arrangement to the degrees of the basis for the module of logarithmic derivation is very interesting to us. Since neither the Poincar\'e polynomial nor the degrees of vector fields are specific to hyperplane arrangements, we investigate here how these things are related in general. For free determinantal arrangements, although the degrees of the basis does not give the factorization of the Poincar\'e polynomial directly, we do find that the Poincar\'e polynomial factors.

In Theorems \ref{CompletePoincare} and \ref{ChordalPoincare}, we show that the Poincar\'e polynomial complement of a free determinantal arrangement factors nicely. We construct a fibration of the complement, then show that the corresponding Serre spectral sequences collapses at the $E_2$ page (which implies that the Poincar\'e polynomial for our complement is the product of the Poincar\'e polynomials of the base and the fiber). In Theorem \ref{FreeHomotopy}, we use the homotopy long exact sequence for our fibration to prove that the higher homotopy groups for the complement are isomorphic to those of $S^3$.

\begin{theorem}
\label{CompletePoincare}Let $G$ be the complete graph on $n$ vertices. Let $U_n = \CC^{2n} \setminus \mathcal{A}_G$, then \emph{$$\textrm{Poin}(U_n,t) = (1+t^3)(1+t)^{n-1}\displaystyle\prod_{k=1}^{n-2}(1+kt).$$}
\end{theorem}

\begin{proof}
We proceed by induction on $n$. For the base case $n=2$, the complement $U_2$ is $\GL(2,\CC)$. Consider the fibration $p: U_2 \rightarrow \CC^2 \setminus \{0\}$, where $p$ is the projection onto the first column of a matrix in $\GL(2,\CC)$, with fibers homotopic to $\CC^2$ minus a line. The base space $\CC^2\setminus \{0\}$ is homeomorphic to $S^3$, and the fiber is homeomorphic to $S^1$. Considering the cohomology Serre spectral sequence, $$ E_2^{p,q} \cong H^p(S^3,H^q(S^1)),$$
we do not have to worry about local coefficients, because $S^3$ is simply connected. Since the target for $d_r:E^{p,q}_r\rightarrow E^{p+r,q-r+1}_r$ is always zero for $r\geq 2$, the spectral sequence collapses at the $E_2$-page. Thus, $$\Poin(U_2,t) = \Poin(S^3,t)\cdot\Poin(S^1,t) = (1+t^3)(1+t).$$

Similarly, we have a fibration $p: U_{n+1}\rightarrow U_n$, where $p$ is the projection onto the first $n$ columns, with fiber $F$ homotopic to $\CC^2$ minus $n$ lines. The cohomology Serre spectral sequence gives us \begin{equation}E_2^{p,q} \cong H^p(U_n, \mathcal{H}^q(F))\Rightarrow H^{p+q}(U_{n+1}).\label{eq:SpecSeq}\end{equation}

To show that we have constant coefficients again, $H^q(F)$, we show that the action of the fundamental group of the base on the homology of the fiber is the identity. Consider the loop $\gamma: [0,2\pi] \rightarrow U_n$, given by
$$\gamma(t) = \left(
                \begin{array}{cccccc}
                  e^{it} & 0 & 1 & 1&\cdots & 1 \\
                  0 & 1 & \frac{1}{2} + e^{-it} &\frac{1}{3} + e^{-it}& \cdots & \frac{1}{n-1}+e^{-it} \\
                \end{array}
              \right).$$
The $(1,2)$-minor of $\gamma$ is $e^{it}$, thus $\gamma$ is a meridian to the subvariety $x_1y_2-x_2y_1=0$. For $j\geq3$, the $(1,j)$-minor is $\frac{1}{j-1}e^{it} + 1$, and all other minors are constant, thus $\gamma$ contracts to a point in the complements of the subvarieties $x_jy_k-x_ky_j=0$ for $j,k\neq 1,2$. We can permute the columns of $\gamma$, to get loops around any particular subvariety $x_jy_k-x_ky_j=0$; thus it is enough to understand the action of $\gamma$ on the homology of the fiber. Since our fiber is the complement of a central arrangement of lines (which is a braid space), elements of $H^1(F)$ generate $H^2(F)$ via the cup product \cite{MR0422674}, hence it is enough to understand how $\gamma$ acts on $H^1(F)$.

Now, denote the columns of $\gamma$ by $v_j$ for $j=1,\ldots, n$. Our fiber is $\CC^2 \setminus \displaystyle\bigcup_{j=1}^n\textrm{span}(v_j)$. We can consider the loops in the fiber given by $\alpha_1=v_1+\varepsilon\left(
                                     \begin{array}{c}
                                       0 \\
                                       e^{i\theta} \\
                                     \end{array}
                                   \right)$ and $\alpha_j=v_j + \varepsilon\left(
                                     \begin{array}{c}
                                       e^{i\theta} \\
                                       0 \\
                                     \end{array}
                                   \right)$ for $j\geq2$ and $0\leq\theta\leq2\pi$. For $\varepsilon$ sufficiently small, the loops $\alpha_j$ are meridians to the lines $\CC v_j$, and can be contracted in the complements $\CC^2\setminus\CC v_k$ for $k\neq j$, therefore they generate $H^1$.

Since $\gamma$ is globally defined on $U_n$ and since $\alpha_j$ at $\gamma(0)$ is defined exactly the same as $\alpha_j$ at $\gamma(1)$, the action of $\gamma$ on $H^1(F)$ is the identity. Thus, in equation (\ref{eq:SpecSeq}), $E_2^{p,q} \cong H^p(U_n,H^q(F)).$

Since $\Var(f)$ has $\left(
                       \begin{array}{c}
                         n+1 \\
                         2 \\
                       \end{array}
                     \right)$ components, $\dim(H^1(U_{n+1})) = \left(
                       \begin{array}{c}
                         n+1 \\
                         2 \\
                       \end{array}
                     \right) = \frac{n(n+1)}{2}.$ Now, $$\dim(E_\infty^{1,0}) + \dim(E_\infty^{0,1}) = \dim H^1(U_{n+1}) = \frac{n(n+1)}{2}.$$ Note that, $E_r^{1,0}$ is not the target of $d_r$ for any $r$, therefore $E_2^{1,0}\cong E_3^{1,0}\cong \cdots \cong E_\infty^{1,0}$. Using the induction hypothesis, we can calculate $\dim(E_\infty^{1,0})$ to be the coefficient of $t$ in $\Poin(U_n,t)$, thus $$ \dim(E_\infty^{1,0}) = (n-1)+\sum_{k=1}^{n-2}k =\frac{(n-1)n}{2}.$$

                     To compute the Poincar\'e polynomial for $F$, we use Theorem \ref{TeraosTheorem}. Note that the module of logarithmic derivations for a central line arrangement is free with a basis consisting of the Euler vector field (which has degree 1), and another of vector field of degree $n-1$ (by Saito's criterion). Thus $\Poin(F,t)=(1+t)(1+(n-1)t)$, which implies that $\dim(E_2^{0,1}) = n$.

                     Now, $$ \frac{n(n+1)}{2}=\dim(E_\infty^{1,0}) + \dim(E_\infty^{0,1})\leq \dim(E_\infty^{1,0})+ \dim(E_2^{0,1}) = \frac{(n-1)n}{2} + n = \frac{n(n+1)}{2},$$ thus we must have $\dim(E_\infty^{0,1}) = \dim(E_2^{0,1})$, and hence $d_r(E_r^{0,1}) = 0$, for all $r\geq 2$.

                     Since elements of $H^1(F)$ generate $H^2(F)$, and since differentials on cup products are derivations, $d_r(E_r^{0,2})=0$ for all $r\geq 2$. Any element of $E_2^{p,q}$ can be written as a linear combination of products of $\alpha\in E_2^{p,0}$ and $\beta\in E_2^{0,q}$, hence $d_2(\alpha \beta) = \beta d_2(\alpha) + \alpha d_2(\beta) = 0$. Inductively, $d_r = 0$ for $r\geq2$, thus $E_2^{p,q} \cong E_\infty^{p,q}$. Furthermore,
                      $$\begin{array}{ccl}
                          \Poin(U_{n+1},t) & = & \Poin(U_n,t)\cdot \Poin(F,t) \\
                           & = & \left((1+t^3)(1+t)^{n-1}\displaystyle\prod_{k=1}^{n-2}(1+kt)\right)\left((1+t)(1+(n-1)t)\right) \\
                           & = & (1+t^3)(1+t)^{n}\displaystyle\prod_{k=1}^{n-1}(1+kt).
                        \end{array}$$

\end{proof}

Following the same proof:

\begin{theorem}
\label{ChordalPoincare}
 Let $G$ be a chordal graph, then Poincar\'e polynomial of $U= \CC^{2n} \setminus \mathcal{A}_G$ factors over $\QQ$ into a product of a cubic with $2|\mathcal{A}_G|-3$ linear terms.
\end{theorem}

\begin{proof}
Since $G$ is chordal, there exists an ordering of vertices, such that for each vertex $v$, the induced subgraph on $v$ and its neighbors that occur before it is a complete graph. Reorganize the columns of $M$ according to this sequence, then we can write $X$ as a fibration similar to the one used in the proof of Theorem \ref{CompletePoincare}.
\end{proof}

Note that our fibration only works when we have a chordal graph. If the graph is not chordal, the fibers are not homotopy equivalent.

\begin{example}
\emph{Consider the cyclic arrangement on $4$ vertices: $f=\Delta_{12}\Delta_{23}\Delta_{34}\Delta_{14}$, we can follow our procedure of projecting the complement onto the first three columns, however some fibers look like $\CC^2$ minus $2$ lines (when the first and third column are linearly independent) and other fibers look like $\CC^2$ minus $1$ line (when the first and third column are linearly dependent).}
\end{example}

When the graph is a chordal, this is no longer an issue since all of the relevant columns are guaranteed to be linearly independent and thus the fibers always look the same. This notion of having homotopic fibers for the complement of the determinantal arrangements is analogous to fiber-type hyperplane arrangements.

This fibration of the complement also gives us information on the homotopy groups.
\begin{theorem}
\label{FreeHomotopy}
 Let $G$ be a chordal graph on $n$ vertices, and $U= \CC^{2n} \setminus \mathcal{A}_G$, then $\pi_{i}(U)\cong\pi_{i}(S^3)$ for $i\geq 2$.
\end{theorem}
\begin{proof}
Without loss of generality, assume that the columns of $M$ are ordered according to the chordal ordering. Let $U_k$ denote the complement of the arrangement restricted to the first $k$ columns. Consider the Serre fibrations $p_k: U_k \rightarrow U_{k-1}$ for $2\leq k\leq n$, where $p_k$ is the projection of $2\times k$ matrices onto the first $k-1$ columns, with fibers $F_k$ homotopic to $\CC^2$ minus $k-1$ lines. For each $k$, consider the homotopy long exact sequence

\begin{equation}0 \leftarrow \pi_0(U_k) \leftarrow \pi_0(F_k) \leftarrow \pi_1(U_{k-1})\leftarrow \pi_1(U_k) \leftarrow \pi_1(F_k) \leftarrow \pi_2(U_{k-1}) \leftarrow \cdots.\label{eq:HLES}\end{equation}

By Proposition 5.6 in \cite{MR1217488}, every central $2$-(hyperplane)arrangement is $K(\pi,1)$, thus for each $k$, $\pi_i(F_k)=0$ for $i\geq2$ and $i=0$. From (\ref{eq:HLES}), $\pi_i(U_k)\cong \pi_i(U_{k-1})$ for $i\geq 3$. Since $U_1=\CC^2\setminus\{0\}\cong S^3$, for each $k$, $\pi_i(U_k)\cong \pi_i(S^3)$ for $i\geq 3$.

Furthermore, consider the segment \begin{equation}\pi_2(U_{k-1})\leftarrow \pi_2(U_k)\leftarrow \pi_2(F_k)\label{eq:HLES2}.\end{equation} When $k=2$, the group on the left in (\ref{eq:HLES2}) is $\pi_2(S^3)= 0$, by induction on $k$, $\pi_2(U_k)=0$ for all $k$.
\end{proof}

\begin{remark}
\emph{Although we have the short exact sequence $$0\rightarrow \pi_1(F_k) \rightarrow \pi_1(U_k) \rightarrow \pi_1(U_{k-1})\rightarrow 0,$$ it is not clear what $\pi_1(U_k)$ is in general.}
\end{remark}

In a survey of hyperplane arrangements, Schenck \cite{MR2933802} posed the problem to define supersolvability for hypersurface arrangements. For hyperplane arrangements, supersolvability is a combinatorial property on the lattice of intersections, and arrangements that are supersolvable are free. In particular, fiber-type arrangements are supersolvable.

For arrangements of more general hypersurfaces, it is not clear whether or not the intersection lattice gives us any useful information, but we can still have fiber-type arrangements. In the context of determinantal arrangements on a $2\times n$ generic matrix, we have fiber-type arrangements when the corresponding graph is chordal. It is tempting to extend this notion to determinantal arrangements on an $m\times n$ generic matrix, however, this cannot be done with our fibration.

\begin{example} \emph{Let the determinantal arrangement $\mathcal{A}$ defined by the product of all maximal minors of a $3\times 7$ generic matrix, and let $U=\CC^{3\times 7}\setminus \mathcal{A}$ be the complement. If we consider the projection $p$ of $U$ onto the first $6$ columns, the fibers are not homotopy equivalent in general. For a generic choice of a basepoint $x$, $p^{-1}(x)$ is the complement of a central generic arrangement of $15$ hyperplanes in $\CC^3$. The fiber $p^{-1}\left(
                                                                       \begin{array}{cccccc}
                                                                         1 & -1 & 0 & 0 & 1 & -1 \\
                                                                         0 & 0 & 1 & -1 & 1 & -1 \\
                                                                         1 & 1 & 1 & 1 & 1 & 1 \\
                                                                       \end{array}
                                                                     \right)$, however, is not the complement of a generic arrangement, thus our projection does not give us a fibration of the complement $U$.}
\end{example}

Although our approach does not extend to generic matrices of larger sizes, it does not mean that a fibration does not exist under certain conditions. We believe that finding such conditions for constructing fibrations would be a start to defining a notion for supersolvability for determinantal arrangements and for hypersurface arrangements in general.

\bibliographystyle{amsplain}
\bibliography{DeterminantalArrangementBib}

\end{document}